\documentclass{amsart}
\usepackage{color}
% ...

% ...

\theoremstyle{plain}
\newtheorem{theorem}{Theorem}[section]

\newtheorem{corollary}[theorem]{Corollary}

\newtheorem{proposition}[theorem]{Proposition}
\newtheorem{example}[theorem]{Example}

\theoremstyle{definition}
\newtheorem{definition}{Definition}[section]

\theoremstyle{remark}
\newtheorem{remark}{Remark}[section]
\newtheorem*{nota*}{Notation}
\newtheorem*{Acknowledgement}{Acknowledgements}

\newcommand{\rk}{\mathrm{rk}}

\newcommand{\CC}{\mathbb{C}}

\newcommand{\caM}{\mathcal{M}}

\title{A note on the simultaneous Waring rank of monomials}

\author{Enrico Carlini}
\address{Department of Mathematical Sciences ``Dipartimento di eccellenza 2018-2022'', Politecnico di Torino, Turin, Italy}
\email{enrico.carlini@polito.it}

\author{Emanuele Ventura}
\address{Department of Mathematics, Texas A\&M University, College Station, U.S.A.}
\email{eventura@math.tamu.edu}



\begin{document}

\maketitle

\date{}

\begin{abstract}

In this paper we study the complex simultaneous Waring rank for collections of monomials. For general collections we provide a lower bound, whereas for special collections we provide a formula for the simultaneous Waring rank. Our approach is algebraic and combinatorial. We give an application to ranks of binomials and maximal simultaneous ranks. Moreover, we include an appendix of scripts written in the algebra software \texttt{Macaulay2} to experiment with simultaneous ranks.

\end{abstract}


\section{Introduction}
\indent The problem of determining a minimal simultaneous decomposition for several homogeneous polynomials dates back to the work of Terracini \cite{Terr}, appeared in 1915, where even the existence of defective cases was observed; for more defective cases, see \cite{CC2003}. More precisely, let $S=\mathbb C[x_0,\ldots,x_n]$ and let $\mathcal F = \lbrace F_1,\ldots,F_s\rbrace$ be a collection of homogeneous polynomials, or {\it forms}, $F_1,\ldots, F_s \in S$.
The {\it simultaneous Waring rank} over $\mathbb C$ of $\mathcal F$, denoted by $\textnormal{rk}_{\mathbb C} \mathcal F$, is
the minimal integer $r$ such that there exists linear forms $L_1,\ldots, L_r$ with the property that every $F_i$ may be written as
$$
F_i = \sum_{j=1}^r \lambda_j L_j^{d_i}, \mbox{ where } d_i = \deg F_i, \lambda_j \in \mathbb C.
$$
\noindent Equivalently, the linear forms $L_j$ minimally decompose simultaneously all the forms $F_i$; note that the degrees $d_i$ need not be necessarily the same. The simultaneous decompositions as above are a direct generalization of the {\it simultaneous diagonalization} of matrices to the context of homogeneous polynomials and, more generally, of tensors. They have recently appeared in statistics, in the context of decompositions of moments \cite[Section 3]{AGHKT}. \\
\indent Terracini's geometric approach led to the notion of {\it Grassmann defectivity} of projective varieties; see \cite[Definition 1.1]{DF} for Veronese surfaces. In his article, Terracini observed that the Veronese variety of dimension $n$ and degree $d$ has $(k,h)$-{\it Grassmann secant variety} filling up all the space if and only if the generic collection of $k+1$ forms of degree $d$ can be written as a linear combination of the powers of $h+1$ linear forms. More recently, Angelini, Galuppi, Mella, and Ottaviani  \cite{AGMO} studied identifiability properties for simultaneous decompositions of general collections.\\
\indent This note is inspired by the problem of explicitly determining the complex simultaneous Waring rank of  special collections of {\it monomials}. Since monomials are the sparsest forms and are particularly amenable to combinatorial approaches, our point of view to investigate this question is mainly algebraic and combinatorial. \\
\indent The structure of this note is the following. In Section 2, we recall a useful lower bound (Proposition \ref{lower bound general}) for a collection of monomials $\mathcal M$. We interpret it combinatorially in Proposition \ref{inclusion-exclusion} as an {\it inclusion-exclusion formula} \cite[Chapter 2]{St} of complex ranks of greatest common divisors of monomials varying in all subsets of $\mathcal M$. In Section 3, we derive a formula for the complex simultaneous Waring rank for special pairs (Theorem \ref{sim Waring for pairs}). This extends to an inclusion-exclusion formula of complex ranks of greatest common divisors for an arbitrary number of special monomials in  Proposition \ref{inclusion-exclusion for c>1} and Proposition \ref{inclusion-exclusion for c=1}. Finally, for collections given by all first derivatives of a given monomial whose exponents are strictly larger than one, we determine the complex simultaneous Waring rank in Proposition \ref{LemmaDerivatives}. In particular, this shows that not for every collection the complex simultaneous Waring rank is an inclusion-exclusion as before. In Section 4, we give an application of the simultaneous Waring rank to ranks of binomials and maximal simultaneous ranks of forms. Finally, in the Appendix we include scripts in the algebra software \texttt{Macaulay2} \cite{M2}, which are hopefully useful to implement our results.

\section{A lower bound}

In this section, we recall a useful lower bound for the simultaneous Waring rank of a collection of special monomials. Before we proceed, we keep the notation from above and we let $T=\mathbb C[X_0,\ldots, X_n]$ act by differentiation on $S$. Given a form $F\in S$,
the {\it apolar ideal} of $F$ is the ideal $F^{\perp} = \lbrace G\in T \ | \ G(F) = 0\rbrace\subset T$. The apolarity lemma \cite[Lemma 1.15]{IK} states that, whenever an ideal of reduced points $I_{\mathbb X}$ is contained in $F^{\perp}$, $F$ admits a Waring decomposition with linear forms dual to the points in $\mathbb X$. In such a case, $\mathbb X$ is an {\it apolar scheme} to $F$. \\
\indent We are now ready to introduce some lower bounds.

\begin{proposition}\label{algorithmic lower bound general}

Let $\mathcal F= \lbrace F_j\rbrace_{j\in J}$ be a collection of forms. The simultaneous Waring rank $\textnormal{rk}_{\mathbb C} \mathcal F$ is at least the dimension of the finite $\mathbb C$-algebra $A=T/I$, where \[I=(L)+\cap_{j\in J} ({F}_j^{\perp}:(L))\] and $L$ is any linear form.

\begin{proof}
Let $I_{\mathbb X}\subset\cap_{j\in J} F_j^\perp$. In order to show the statement, we have to establish a lower bound on the cardinality of $\mathbb{X}$. We adapt the proof of \cite[Theorem 3.3]{CCCGW}. Note that the simultaneous rank is at least the dimension of the algebra
\[T/(I_{\mathbb X}:(L)+L)\]
for any linear form $L$.
Let $J=\cap_{j\in J} F_j^\perp$. Since $I_{\mathbb X}\subset J$, we have that
\[I_{\mathbb X}:(L)+(L)\subset J:(L)+(L)\]
and the conclusion follows.
\end{proof}
\end{proposition}

Note that Proposition \ref{algorithmic lower bound general} provides a tool to experiment with the lower bound on any collection of forms; this is made possible by the first script of the Appendix. Moreover, for special collections of form we are able to derive more explicit lower bounds.

\begin{proposition}\label{lower bound general}

Let $\mathcal M= \lbrace M_j\rbrace_{j\in J}$ be a collection of monomials, where each monomial is of the form $M_i=x_0^{a_{0,j}}x_1^{a_{1,j}}\cdots x_n^{a_{n,j}}$, and $1\leq a_{i,j}$ for every $i,j$. The simultaneous Waring rank $\textnormal{rk}_{\mathbb C} \mathcal M$ is at least the dimension of the finite $\mathbb C$-algebra $A=T/I$, where $I=(X_i)+\cap_{j\in J} M_j^{\perp}$ for any $i$.

\begin{proof}
By Proposition \ref{algorithmic lower bound general}, the sum of the Hilbert function in all degrees of the ideal
\[I=(X_i)+(\cap_{j\in J} M_j^{\perp}): (X_i)=(X_i)+\cap_{j\in J} M_j^{\perp}: (X_i) = (X_i)+\cap_{j\in J} M_j^{\perp}\] gives a lower bound to the degree of an ideal of reduced points in $\cap_{j\in J} M_j^{\perp}$.
\end{proof}
\end{proposition}

As we have established a lower bound for the simultaneous Waring rank of monomials, let us record the following observation concerning an upper bound.

\begin{remark}
Let $\mathcal M = \lbrace M_1,\ldots, M_s\rbrace$ and let $M = \textnormal{lcm}(M_1,\ldots,M_s)$ be their least common multiple. We have $\textnormal{rk}_{\mathbb C} \mathcal M \leq  \textnormal{rk}_{\mathbb C} M$. Indeed, $M^{\perp}\subseteq M_i^{\perp}$, for each $1\leq i\leq s$; hence any ideal of reduced points  $I_{\mathbb X}\subset M^{\perp}$ would be contained in each $M_i^{\perp}$ as well, thus providing a simultaneous decomposition. This upper bound is far from being optimal in general; see Example \ref{ex1}.
\end{remark}

By a further specialization of our collection we can give am explicit combinatorial description of the dimension of the algebra $A$ of Proposition \ref{lower bound general}.

\begin{proposition}\label{inclusion-exclusion}

Let $\mathcal M= \lbrace M_j\rbrace_{j\in J}$, where $M_j=x_0^{a_{0,j}}x_1^{a_{1,j}}\cdots x_n^{a_{n,j}}$, and $1\leq a_{0,j}\leq a_{i,j}$ for every $i,j$. Let $M_{j_1,\ldots, j_k}$ denote the greatest common divisor of $M_{j_1}, \ldots, M_{j_k}$. The dimension of the finite $\mathbb C$-algebra $A=T/I$, where $I=(X_0)+\cap_{j\in J} M_j^{\perp}$ is given by the alternating sum
$$
\dim_{\mathbb C} A = \sum_{\emptyset \neq [k] \subset J} (-1)^{k+1} \textnormal{rk}_{\mathbb C} M_{j_1,\ldots, j_k}.
$$

\begin{proof}

We show that the {\it standard monomials} of $A$, that is, the monomials {\it not} in $I$, are exactly the monomials dividing at least one of the monomials $\tilde{M}_j=X_1^{a_{1,j}}\cdots X_n^{a_{n,j}}$. Suppose $N$ is a monomial dividing $\tilde{M}_j$ for some $j$. Then $N$ cannot be a monomial in $I$, since it does not annihilate $\tilde{M}_j$. Conversely, suppose $N\notin I$. Hence there exists $\tilde{M}_j$ such that $N$ does not annihilate it. This implies that all the exponents of $N$ are at most those of $\tilde{M_j}$. In other words, $N$ is a divisor of $\tilde{M}_j$. \\
\indent Thus the set of standard monomials coincides with the divisors of all of the monomials $\tilde{M}_j=X_1^{a_{1,j}}\cdots X_n^{a_{n,j}}$. The number of these divisors is given by the alternating sum $\sum_{\emptyset \neq [k] \subset J} (-1)^{k+1} \textnormal{rk}_{\mathbb C} M_{j_1,\ldots, j_k}$. Indeed, the number of divisors of $\tilde{M}_j$ is $\textnormal{rk}_{\mathbb C} M_j$. Moreover, the number of common divisors of $M_{j_1}, \ldots, M_{j_k}$ is $\textnormal{rk}_{\mathbb C} M_{j_1,\ldots, j_k}$. Applying the inclusion-exclusion principle \cite[Chapter 2.1]{St} we finish the proof.
\end{proof}

\end{proposition}

\section{Special collections of monomials}

We begin with special pairs of monomials and we determine their complex simultaneous Waring rank.

\begin{proposition}\label{upperbound for pairs}

Let $\mathcal M=\lbrace M_1, M_2\rbrace$, $M_1={x_0}^{c}x_1^{a_1}\cdots x_n^{a_n}$ and $M_2={x_0}^{c}x_1^{b_1}\cdots x_n^{b_n}$ with $1\leq c \leq a_i$ and $1\leq c \leq b_i$, for $i=1,\ldots,n$. Let $M_{12} = \textnormal{gcd}(M_1,M_2)$ be the greatest common divisor of $M_1,M_2$. Suppose $| a_ i - b_i | = 0,1$ or $| a_ i - b_i |\geq c+1$. Then
$$
\textnormal{rk}_{\mathbb C} \mathcal M \leq \textnormal{rk}_{\mathbb C} M_1+\textnormal{rk}_{\mathbb C} M_2 - \textnormal{rk}_{\mathbb C} M_{12}.
$$

\begin{proof}
For each $i>0$ consider the ideals $A_i=(X_0^{c+1},X_i^{a_i+1})$ and $B_i=(X_0^{c+1},X_i^{b_i+1})$. Our argument is for each fixed $i$. \\
\indent Let us fix an arbitrary $i=1,\ldots,n$ and suppose $a_i\leq b_i$. Let $F_i\in (A_i)_{a_i+1}$ be a square-free element and note that all monomials of $F_i$ belong to $B_i$ except possibly a scalar multiple of $X_i^{a_i+1}$. Now consider the colon ideal
\[J_i=B_i:(X_i^{a_i+1})=(X_0^{c+1},X_i^{b_i-a_i}).\]
Since $| a_ i - b_i | = 0,1$ or $| a_ i - b_i |\geq c+1$, there exists $H_i\in J_i$ square-free of minimal degree with no common factors with $F_i$. Indeed, if $a_i-b_i = 0$, we may take $H_i$ to be a scalar. If $b_i - a_i = 1$, we may take $H_i = X_i$, as $F_i$ is general and hence avoids such a zero. If $b_i-a_i \geq c+1$, we may take $H_i = X_0^{b_i-a_i}-X_i^{b_i-a_i}$, as $F_i$ is general and hence avoids such zeros. Thus $G_i=F_iH_i$ is a square-free element of $(B_i)_{b_i+1}$. \\
\indent The case $a_i>b_i$ is analogous. Namely, we pick $F_i\in (B_i)_{b_i+1}$ and construct $G_i=F_iH_i$ square-free in $(A_i)_{a_i+1}$.\\
\indent In conclusion, for each $i>0$, we construct square-free binary forms $F_i\in A_i$ and $G_i\in B_i$ such that one of the two divides the other. Now, let us consider the ideals
\[I_1=(F_1,\ldots,F_n)\subset M_1^\perp\mbox{ and }I_2=(G_1,\ldots,G_n)\subset M_2^\perp.\]
Note that $I_1=I(\mathbb{X}_1)$ and $I_2=I(\mathbb{X}_2)$, where $\mathbb{X}_i$ is a minimal apolar scheme of points for $M_i$. By construction $\mathbb{X}_1\cup\mathbb{X}_2$ is apolar to $M_1$ and $M_2$. Moreover, the ideal of $\mathbb{X}_1\cap\mathbb{X}_2$ is precisely
\[I_1+I_2=\left( \textnormal{gcd}(F_i,G_i):i=1,\ldots,n\right).\]
The scheme $\mathbb{X}_1\cap\mathbb{X}_2$ is a minimal apolar scheme for $M_{12}$. Hence $|\mathbb{X}_1\cup\mathbb{X}_2|=\mathrm{rk}_{\mathbb C} M_1+\mathrm{rk}_{\mathbb C} M_2-\mathrm{rk}_{\mathbb C} M_{12}$. This completes the proof.
\end{proof}

\end{proposition}

\begin{theorem}\label{sim Waring for pairs}

Let $\mathcal M=\lbrace M_1, M_2\rbrace$, $M_1={x_0}^{c}x_1^{a_1}\cdots x_n^{a_n}$ and $M_2={x_0}^{c}x_1^{b_1}\cdots x_n^{b_n}$, with $1\leq c \leq a_i$ and $1\leq c \leq b_i$, for $i=1,\ldots,n$. Let $M_{12} = \textnormal{gcd}(M_1,M_2)$ be the greatest common divisor of $M_1,M_2$. Suppose $| a_ i - b_i | = 0,1$ or $| a_ i - b_i |\geq c+1$. Then the simultaneous Waring rank satisfies
$$
\textnormal{rk}_{\mathbb C} \mathcal M = \textnormal{rk}_{\mathbb C} M_1+\textnormal{rk}_{\mathbb C} M_2 - \textnormal{rk}_{\mathbb C} M_{12}.
$$

\begin{proof}
Proposition \ref{upperbound for pairs} gives the upper bound. We need to show the lower bound. Proposition \ref{lower bound general} and Proposition \ref{inclusion-exclusion} give the desired lower bound.
\end{proof}

\begin{remark}

In the statement of Theorem \ref{sim Waring for pairs} the conditions $| a_ i - b_i | = 0,1$ or $| a_ i - b_i |\geq c+1$ are always satisfied for $c=1$.

\end{remark}

\end{theorem}

The following example shows how the construction in Proposition \ref{upperbound for pairs} gives a minimal simultaneous decomposition for special pairs.

\begin{example}\label{ex1}
We keep the notation of Proposition \ref{upperbound for pairs} and Theorem \ref{sim Waring for pairs}. Let $\mathcal M = \lbrace M_1, M_2 \rbrace$, where $M_1 = x_0x_1^3x_2^4x_3^7$ and $M_2=x_0x_1^4x_2^2x_3^5$. In this case, we have $c=1$ and $M_{12} = x_0x_1^3x_2^2x_3^5$. We consider the ideals
$A_1=(X_0^2,X_1^4)$, $B_1=(X_0^2,X_1^5)$, $A_2=(X_0^2,X_2^5)$, $B_2=(X_0^2,X_2^3)$, $A_3=(X_0^2,X_3^8)$, $B_3=(X_0^2,X_3^6)$. We can choose general square-free binary forms $F_1\in A_1, F_2\in B_2, F_3\in B_3$, so that, for $H_1 = X_1$, $H_2 = X_2^2-X_0^2$, $H_3 = X_3^2-X_3^2$, the binary forms
$G_1 = F_1H_1\in B_1, G_2 = F_2H_2\in A_2, G_3 = F_3H_3\in A_3$ are square-free. The zeros of $G_1, G_2, G_3$ provides the minimal apolar scheme to the collection $\mathcal M$. Let $M = \textnormal{lcm}(M_1,M_2) = x_0x_1^4x_2^4x_3^7$. Thus $\textnormal{rk}_{\mathbb C} \mathcal M \leq \textnormal{rk}_{\mathbb C} M = 200$. On the other hand, we have $\textnormal{rk}_{\mathbb C} \mathcal M = \textnormal{rk}_{\mathbb C} M_1 + \textnormal{rk}_{\mathbb C} M_2 - \textnormal{rk}_{\mathbb C} M_{12} = 178$.
\end{example}

We introduce two definitions in order to extend Theorem \ref{sim Waring for pairs} to more general collections of monomials; these definitions are crucial to apply our proof techniques.

\begin{definition}[{\bf (1,1)-free collections}]\label{(1,1)free}

Let $\mathcal M= \lbrace M_j\rbrace_{j\in J}$ be a set of monomials, where $M_j=x_0x_1^{a_{1,j}}\cdots x_n^{a_{n,j}}$.  For each $1\leq i\leq n$, let us order linearly the exponents of the variable $x_i$ appearing in the monomials $M_j\in \mathcal M$; let $\mathcal E_i = \lbrace a_{i,k_1}\leq a_{i,k_2}\leq \cdots \leq a_{i,k_{|J|}}\rbrace$ be the set of those exponents. Let $\mathcal D_i$ denote the multiset of differences $\lbrace a_{i,k_{s+1}} - a_{i,k_{s}} \rbrace$ of successive elements in $\mathcal E_i$. The collection $\mathcal M$ is said to be $(1,1)$-{\it free} if, for every $i\in \lbrace 1,\ldots, n \rbrace$, the multiset $\mathcal D_i$ contains $1$ at most one time.

\end{definition}

\begin{definition}[{\bf Free collections}]\label{Free collection}

Let $\mathcal M= \lbrace M_j\rbrace_{j\in J}$ be a set of monomials, where $M_j=x_0^cx_1^{a_{1,j}}\cdots x_n^{a_{n,j}}$. As in Definition \ref{(1,1)free}, let $\mathcal D_i$ denote the multiset of differences  $\lbrace a_{i,k_{s+1}} - a_{i,k_{s}} \rbrace$ of successive elements in $\mathcal E_i$. The collection $\mathcal M$ is said to be {\it free} if $\mathcal M$ is $(1,1)$-free and, for every $i$, the multiset $\mathcal D_i$ does not contain nonzero integers smaller than or equal to $c$.

\end{definition}

\begin{proposition}\label{inclusion-exclusion for c>1}

Let $\mathcal M= \lbrace M_j\rbrace_{j\in J}$, where $M_j=x_0^cx_1^{a_{1,j}}\cdots x_n^{a_{n,j}}$, with $a_{i,j}\geq c$, for every $i,j$. Assume $\mathcal M$ is a free collection. Let $M_{j_1,\ldots, j_k} = \textnormal{gcd}(M_{j_1}, \ldots, M_{j_k})$ be the greatest common divisor of $M_{j_1}, \ldots, M_{j_k}$. The simultaneous Waring rank satisfies
$$
\textnormal{rk}_{\mathbb C} \mathcal M = \sum_{\emptyset \neq [k] \subset J} (-1)^{k+1} \textnormal{rk}_{\mathbb C} M_{j_1,\ldots, j_k}.
$$

\begin{proof}
Proposition \ref{lower bound general} and Proposition \ref{inclusion-exclusion} give the lower bound. We show the upper bound constructing
an minimal apolar scheme of points $\mathbb X$. For each $1\leq i\leq n$, let us order linearly the exponents of the variable $x_i$ appearing in the monomials $M_j\in \mathcal M$; let $\mathcal E_i = \lbrace a_{i,k_1}\leq a_{i,k_2}\leq \cdots \leq a_{i,k_{|J|}}\rbrace$ be the set of those exponents. Since the collection $\mathcal M$ is {\it free}, for every $1\leq i\leq n$, the multiset $\mathcal D_i$ of differences of successive elements in $\mathcal E_i$ is $(1,1)$-free and does not contain nonzero integers smaller than or equal to $c$. For each $1\leq j \leq |J|$, the set $\mathbb X_i^j$ denotes the set of $x_i$-coordinates of points apolar to $M_j$. The procedure presented in the proof of Proposition \ref{upperbound for pairs} produces the sets $\mathbb X_i^j$, such that $\mathbb X_i^{k_1}\subseteq \mathbb X_i^{k_2}\subseteq \cdots \subseteq \mathbb X_i^{k_j}\subseteq \cdots \subseteq \mathbb X_i^{k_{|J|}}$, for every $i$. Since the collection $\mathcal M$ is {\it free}, the construction gives a set of points $\mathbb X$, all of which differ in at least one coordinate.\\
\indent Thus for each monomial $M_j$ we have a collection of points $\mathbb X_j$ that constitute a minimal apolar scheme for $M_j$. Hence $|\mathbb X_j| = \textnormal{rk}_{\mathbb C} M_j$. The intersection $\mathbb X_{j_1}\cap \cdots \cap \mathbb X_{j_k}$ is a minimal collection of points apolar to $M_{j_1,\ldots,j_k}$.
Since the points in the union $\mathbb X = \bigcup_{j\in J} \mathbb X_j$ decompose all monomials in $\mathcal M$, $\textnormal{rk}_{\mathbb C} \mathcal M \leq |\bigcup_{j\in J} \mathbb X_j|$. By the inclusion-exclusion principle \cite[Chapter 2.1]{St}, the cardinality of the union of the sets $\mathbb X_j$ is given by the alternating sum
$$
\sum_{\emptyset \neq [k] \subset J} (-1)^{k+1} |\mathbb X_{j_1}\cap \cdots \cap \mathbb X_{j_k}| = \sum_{\emptyset \neq [k] \subset J} (-1)^{k+1} \textnormal{rk}_{\mathbb C} M_{j_1,\ldots, j_k}.
$$
\noindent This concludes the proof.
\end{proof}

\end{proposition}

\begin{remark}

The assumption on the collection $\mathcal M$ being free is necessary for our construction. For instance, if the $j$th multiset of differences $\mathcal D_j$ contains $1$ two times for some $j$, the construction in the proof of Proposition \ref{inclusion-exclusion for c>1} does not produce distinct points. Hence, when a collection is not free, we do not know whether this upper bound is valid or not.

\end{remark}

\begin{corollary}\label{inclusion-exclusion for c=1}
Let $\mathcal M= \lbrace M_j\rbrace_{j\in J}$, where $M_j=x_0x_1^{a_{1,j}}\cdots x_n^{a_{n,j}}$, with $a_{i,j}\geq 1$ for every $i,j$.  Assume $\mathcal M$ is a $(1,1)$-free collection. Let $M_{j_1,\ldots, j_k} = \textnormal{gcd}(M_{j_1}, \ldots, M_{j_k})$ be the greatest common divisor of $M_{j_1}, \ldots, M_{j_k}$. The simultaneous Waring rank satisfies
$$
\textnormal{rk}_{\mathbb C} \mathcal M = \sum_{\emptyset \neq [k] \subset J} (-1)^{k+1} \textnormal{rk}_{\mathbb C} M_{j_1,\ldots, j_k}.
$$
\end{corollary}

\indent We now extend Theorem \ref{sim Waring for pairs} to monomials having, possibly, different supports.

\begin{theorem}\label{notsamesupportpairs}

Let $M_1=x_0y_1^{a_1}\cdots y_s^{a_s}z_1^{b_1}\cdots z_r^{b_r}$ and $M_2=x_0y_1^{c_1}\cdots y_s^{c_s}t_1^{d_1}\cdots t_{\ell}^{d_{\ell}}$. Here $b_i, d_i\geq 2$, and $a_i, c_i \geq 0$. Let $M_{12}=\textnormal{gcd}(M_1,M_2)$ be the greatest common divisor of $M_1$ and $M_2$. The simultaneous Waring rank of $\mathcal M = \lbrace M_1, M_2\rbrace$ satisfies
$$
\textnormal{rk}_{\mathbb C} \mathcal M = \textnormal{rk}_{\mathbb C} M_1+\textnormal{rk}_{\mathbb C} M_2 - \textnormal{rk}_{\mathbb C} M_{12}.
$$
\begin{proof}
As in Theorem  \ref{sim Waring for pairs}, the lower bound is deduced from a variation of Proposition \ref{lower bound general}: the algebra $A = \mathbb C[X_0,Y_{s}, Z_{r}, T_{\ell}]/I$, where $I = (X_0)+M_1^{\perp}\cap M_2^{\perp}$, has dimension
$\textnormal{rk}_{\mathbb C} M_1+\textnormal{rk}_{\mathbb C} M_2 - \textnormal{rk}_{\mathbb C} M_{12}$. Indeed, the monomials not in $I$ are formed by the union of those that are not in $(X_0)+M_1^{\perp}$ and those that are not in $(X_0)+M_2^{\perp}$. These two sets have in common the monomials not divisible by $X_0$ and dividing $M_{12}$. Hence the dimension of the finite $\mathbb C$-algebra $A$ is $\textnormal{rk}_{\mathbb C} M_1+\textnormal{rk}_{\mathbb C} M_2 - \textnormal{rk}_{\mathbb C} M_{12}$. To see the upper bound, for both $M_1$ and $M_2$ we construct minimal apolar schemes containing the points $[1: \alpha_{1,i} : \cdots : \alpha_{s,i} : 0 : \cdots : 0]$ (here the number of the last zeros is $r+\ell$), where $\sum_{i=1}^{\min\lbrace a_i, c_i\rbrace+1} \alpha_{k,i} = 0$ for each $1\leq k\leq s$. Indeed, these points form a minimal apolar scheme for $M_{12}$. Moreover, they can be extended to minimal apolar schemes for both $M_1$ and $M_2$ by the assumption $b_i, d_i\geq 2$. Taking the union of those minimal apolar schemes for $M_1$ and $M_2$ yields the result.
\end{proof}
\end{theorem}

\begin{remark}
In Theorem \ref{notsamesupportpairs}, assuming instead $b_i = 1$ (or $d_i=1$) for some $i$, we do not know whether the upper bound is valid or not.
\end{remark}

\begin{corollary}
Let $\mathcal M= \lbrace M_1, M_2\rbrace$, $M_1=x_0x_1^{a_1}\cdots x_n^{a_n}$ and $M_2=x_0y_1^{b_1}\cdots y_s^{b_s}$, with $a_i,b_i\geq 2$. The simultaneous Waring rank of $\mathcal M$ satisfies
$$
\textnormal{rk}_{\mathbb C} \mathcal M = \textnormal{rk}_{\mathbb C} M_1+\textnormal{rk}_{\mathbb C} M_2 - 1.
$$
\end{corollary}

\indent A natural case is when the collection of monomials $\mathcal M$ is given by all the first partial derivatives of a fixed monomial $M$.

\begin{proposition}\label{LemmaDerivatives}
Let $M=x_0^{a_0}\cdots x_n^{a_n}$ be a monomial with $a_i>1$ for all $i$. Let  $\mathcal M = \{\partial_{x_0}M,\ldots,\partial_{x_n}M\}$ be the collection of its first derivatives. Then
 $$\rk_{\CC}M = \rk_{\CC}\caM.$$

\begin{proof}
Let $H=\bigcap_{i=0}^n (\partial_{x_i} M)^{\perp}$ and $d=\deg M$. Since $M^{\perp}+T_d = H$, by the apolarity lemma \cite[Lemma 1.15]{IK}, we have $\rk_{\CC}\caM \leq \rk_{\CC}M$. Again, by the apolarity lemma, proving the equality is now equivalent to showing that every reduced ideal of points in $H$ is not supported on strictly less than $\rk_{\CC} M$ points. Let us set $F=x_0^{a_0-1}x_1^{a_1}\cdots x_n^{a_n}$. Note that $\rk_{\CC} F=\rk_{\CC} M$ and, by definition of $H$, $H\subset F^{\perp}$. This concludes the proof.
\end{proof}

\end{proposition}


\section{Final remarks}

\noindent {\bf Sub-additivity of the rank and rank of binomials.} From the definition, it follows that $\rk_{\CC} (F_1+\ldots+F_r)\leq \rk_{\CC} F_1+\ldots+\rk_{\CC} F_r$. The well-known Strassen's conjecture states that, whenever the forms are in disjoint sets of variables, then equality holds; see \cite{Strassen1973,CCO2017, T2015}. The simultaneous Waring rank enters the picture because
\[\rk_{\CC} (F_1+\ldots+F_r)\leq \rk_{\CC} \lbrace F_1,\ldots, F_r \rbrace \leq\rk_{\CC} F_1+\ldots+\rk_{\CC} F_r.\]

\noindent In particular, whenever the simultaneous rank is strictly less than the sum of the ranks, the rank is strictly sub-additive.

As a direct consequence of Theorem \ref{sim Waring for pairs}, we obtain an upper bound for the rank of special binomials.
\begin{corollary}\label{rankofbinomial}

Let $M_1, M_2, M_{12}$ be as in Theorem  \ref{sim Waring for pairs}. Then \[\textnormal{rk}_{\mathbb C}(M_1+M_2)\leq \textnormal{rk}_{\mathbb C} M_1+\textnormal{rk}_{\mathbb C} M_2 - \textnormal{rk}_{\mathbb C} M_{12}.\]

\end{corollary}

\begin{example}
Let $M_1=x_0x_1x_2x_3^2$ and $M_2=x_0x_1x_2^2x_3$, we have $\rk_{\CC} M_1=\rk_{\CC} M_2=12$. Thus, $\rk_{\CC} (M_1+M_2)\leq 24$. Corollary \ref{rankofbinomial} implies the improved upper bound
\[\rk_{\CC} (M_1+M_2)\leq 12+12-8=16.\]
However, in a private communication, Bruce Reznick informed us that $\rk_{\CC} (M_1+M_2)\leq 15$. Hence this bound is not sharp.
\end{example}

\noindent {\bf High simultaneous rank.}
\noindent The study of {\em maximal} Waring ranks compared with {\em generic} Waring ranks has attracted a lot of attention; see, for instance, \cite{BT2015,BT2016,BHMT2017}. A similar analysis could be performed in the case of the simultaneous Waring ranks. Here several difficulties arise due, for example, to the fact that generic simultaneous ranks are largely unknown.

\begin{remark}
The simultaneous rank for a generic collection of $k+1$ forms of degree $d$, in $n+1$ variable, is $h+1$ if and only if the Segre-Veronese variety $\mathbb{P}^k\times\mathbb{P}^n$ embedded in bidegree $(1,d)$ is such that its $h$-secant variety fills up the ambient space. Indeed, the simultaneous rank for such a generic collection can be thought of the rank of the generic point of the projective $k$-dimensional linear space spanned by those $k+1$ forms of degree $d$. Thus the knowledge of the generic simultaneous rank relates to the open problem of classifying defective Segre-Veronese varieties; see, for example, \cite{BBC2012}.
\end{remark}

In \cite{CCG}, it is shown that monomials in three variables provide examples of forms having complex rank higher than the generic rank. An analogue result is the following:

\begin{proposition}
Let $S=\CC[x_0,x_1,x_2]$. Then
\[\rk_{\CC} \lbrace x_0x_1^tx_2^{t+1},x_0x_1^{t+1}x_2^t\rbrace=t^2+4t+3,\]
\[\rk_{\CC} \lbrace x_0x_1^tx_2^{t+2},x_0x_1^{t+2}x_2^t\rbrace=t^2+6t+5.\]
\noindent For $t\geq 1$, these ranks are strictly higher than $\rk_{\CC} \lbrace F_1,F_2\rbrace$, for generic $F_1,F_2\in S_d$.
\end{proposition}
\begin{proof}
The rank for the generic pair $F_1,F_2\in S_d$ is
$$
\left \lceil \frac{1}{2} \binom{d+2}{2} \right \rceil
$$
\noindent unless $d=3$, in which case the $5$th secant variety has codimension one, rather than filling up the ambient space; see \cite[Theorem 1.3]{BD2010}.
For $d=2t+2$, the generic rank is
\[\left\lceil t^2+\frac{7}{2}t+3\right\rceil.\]
\noindent For $d=2t+3$, the generic rank is
\[\left\lceil t^2+\frac{9}{2}t+5\right\rceil.\]
The result follows by a direct computation.
\end{proof}

\section{Appendix}

\begin{verbatim}

------ We use Apolarity script available at https://github.com/zteitler/ApolarIdeal.m2.
------ Script 1. Input: List of homogeneous polynomials of arbitrary degrees.
------ Output: Lower bound for the simultaneous Waring rank, by Prop. 2.1.

Lowerbound = method()
Lowerbound (List):=(L) -> (
    R:= ring(first L);
    J = sub(ideal(1),R);
    for i from 0 to #L-1 do
    J = intersect(J,Apolar(L_i));
    l = random(1,R);
    I = ideal(l)+(J:ideal(l));
    D = {};
    for f in L do D=append(D,degree(f));
    s = (max D)_0;
    lb=sum for i from 0 to s list hilbertFunction(i,I);
    return lb
    );

------ Script 2. Input: A monomial.
------ Output: Its complex Waring rank.

CWMon = method()
CWMon (RingElement):=(m)-> (
    Exponents=exponents(m);
    if(#Exponents > 1 ) then (
    error "Expected monomial as input";
  );
   Exponents=(exponents(m))_0;
   Exponents1={};
   for i in Exponents do if (i!=0) then Exponents1=append(Exponents1,i+1);
   return product(Exponents1)/(min(Exponents1))
   );

------ Script 3. Input: A monomial.
------ Output: Position(s) of the minimal exponent(s).

Minspot = method()
Minspot(RingElement):=(m)-> (
     if (#exponents(m)>1) then (
    error "Expected a monomial as input";
  );
    Minspt = {};
    for i from 0 to #(exponents(m))_0-1 do (
	if ( (exponents(m))_0_i == min((exponents(m))_0)) then
    Minspt = append(Minspt,i+1)
    );
    return set Minspt
    );

------ Script 4. Input: A list of monomials.
------ Output: Alternating sum of the complex Waring ranks of their gcd.

AlternatingSum = method()
AlternatingSum(List):=(L) -> (
t=#L;
  Alt = {};	
 for k from 1 to t do (
	Subsets = subsets(t,k);
    for S in Subsets do (
	    SList = toList S;
	    G = L_SList_0;
	    for j from 1 to #SList-1 do G = gcd(G,L_SList_j);
	    rk = (-1)^(k+1)*CWMon(G);
	 Alt = append(Alt, rk);
	    );
	);
    return sum(Alt)	
);

------ Script 5. Input: A list of monomials.
------ Output: Dimension of the algebra appearing in Proposition 2.3.

DimAlg = method()
DimAlg(List):=(L) -> (
    R:= ring(first L);
    for m in L do
    if (#exponents(m)>1) then (
    error "Expected list of monomials as input";
  );
for m in L do
    if (min((exponents(m))_0)==0) then (
    error "Expected list of monomials with full support";
  );
Min = {};
for m in L do Min=append(Min, Minspot(m));
Int = Min_0;
for i from 1 to #L-1 do Int = Int*Min_i;
if (#Int == 0) then (error "Monomials do not satisfy hypothesis of Prop. 2.3";
);
IntList = toList Int;
r = IntList_0;
t = ideal (vars R)_{r-1};
J = sub(ideal(1),R);
for i from 0 to #L-1 do J = intersect(J,Apolar(L_i));
I = t+(J:t);
D = {};
for m in L do D=append(D,degree(m));
s = (max D)_0;
da = sum for i from 0 to s list hilbertFunction(i,I);
return da;
);	

------ If L satisfies Prop. 2.3, then

DimAlg(L)==AlternatingSum(L)

------ is true

------ Script 6. Input: A list of monomials.
------ Output: Check if the list is a free collection, according to Def. 3.2.

Checkfree = method()
Checkfree(List):=(L) -> (
R:= ring(first L);
Min = {};
for m in L do Min=append(Min, Minspot(m));
Int = Min_0;
for i from 1 to #L-1 do Int = Int*Min_i;
if (#Int == 0) then (error "Monomials do not satisfy hypothesis of Prop. 3.4";
);
IntList = toList Int;
r = IntList_0;
T = {};
for m in L do T = append(T,(exponents(m))_0_(r-1));
Tset = set T;
if (#Tset != 1) then (error "Monomials do not satisfy hypothesis of Prop. 3.4";
    );
c = T_0;
v =  numgens R-1;
for i from 0 to v do (
    Exp = {};
    for j from 0 to #L-1 do Exp = append(Exp,(exponents(L_j))_0_i);
    Exp = sort Exp;
    Diff =  {};
    for k from 0 to #Exp-2 do Diff = append(Diff, Exp_(k+1)-Exp_k);
    Num1 =  {};
    Numc =  {};
    for l from 0 to #Diff-1 do (
    if (Diff_l == 1) then Num1=append(Num1,l)
    else if (Diff_l>c) then Numc=append(Numc,l)
    );
if (#Num1 > 1) then (error "Monomials do not satisfy hypothesis of Prop. 3.4";
    );
if (c!=1 and #Numc > 0) then (error "Monomials do not satisfy hypothesis of Prop. 3.4";
    );
);
return print "The collection of monomials is free"
);
\end{verbatim}
\vspace{1cm}

\begin{Acknowledgement}

The authors would like to thank the referee for valuable comments and suggestions. The first author is a member of GNSAGA a group of INDAM. The first author acknowledges the financial support of the Politecnico di Torino.

\end{Acknowledgement}

\nocite{*}
\bibliographystyle{amsplain}
\bibliography{Biblio}

\providecommand{\bysame}{\leavevmode\hbox to3em{\hrulefill}\thinspace}
\providecommand{\MR}{\relax\ifhmode\unskip\space\fi MR }
\providecommand{\MRhref}[2]{%
  \href{http://www.ams.org/mathscinet-getitem?mr=#1}{#2}
}
\providecommand{\href}[2]{#2}
\begin{thebibliography}{10}

\bibitem{AGHKT}
A.~Anandkumar, D.~Hsu R.~Ge, S.~M. Kakade, and M.~Telgarsky, \emph{Tensor
  decompositions for learning latent variable models}, Journal of Machine
  Learning Research \textbf{15} (2014), 2773--2832.

\bibitem{AGMO}
E.~Angelini, F.~Galuppi, M.~Mella, and G.~Ottaviani, \emph{On the number of
  {W}aring decompositions for a generic polynomial vector}, J. Pure Appl.
  Algebra (2017).

\bibitem{BBC2012}
E.~Ballico, A.~Bernardi, and M.V.Catalisano, \emph{Higher secant varieties of
  $\mathbb{P}^n\times\mathbb{P}^1$ embedded in bidegree $(a,b)$}, Commun.
  Algebra \textbf{40} (2012), no.~10, 3822--3840.

\bibitem{BD2010}
K.~Baur and J.~Draisma, \emph{Secant dimensions of low-dimensional homogeneous
  varieties}, Adv. Geom. \textbf{10} (2015), no.~1, 1--29.

\bibitem{BT2015}
G.~Blekherman and Z.~Teitler, \emph{On maximum, typical and generic ranks},
  Math. Ann. \textbf{362} (2015), no.~3-4, 1021--1031.

\bibitem{BHMT2017}
J.~Buczy\'{n}ski, K.~Han, M.~Mella, and Z.~Teitler, \emph{On the locus of
  points of high rank},  (2017), arXiv: 1703.02829.

\bibitem{BT2016}
J.~Buczy\'{n}ski and Z.~Teitler, \emph{Some examples of forms of high rank},
  Collect. Math. \textbf{67} (2016), no.~3, 431--441.

\bibitem{CCG}
E.~Carlini, M.V. Catalisano, and A.V. Geramita, \emph{The solution to
  {W}aring's problem for monomials and the sum of coprime monomials}, J.
  Algebra \textbf{370} (2012), 5--14.

\bibitem{CCO2017}
E.~Carlini, M.V. Catalisano, and A.~Oneto, \emph{{W}aring loci and {S}trassen
  conjecture}, Adv. Math. \textbf{314} (2017).

\bibitem{CCCGW}
E.~Carlini, L.~Chiantini, M.V. Catalisano, A.V. Geramita, and J.~Woo,
  \emph{Symmetric tensors: rank, strassen’s conjecture and
  $e$-computability}, ANN SCUOLA NORM-SCI \textbf{XVIII} (2018), 363--390.

\bibitem{CC2003}
E.~Carlini and J.~Chipalkatti, \emph{On {W}aring's problem for several
  algebraic forms}, Comment. Math. Helv. \textbf{78} (2003), no.~3, 494--517.

\bibitem{DF}
C.~Dionisi and C.~Fontanari, \emph{{G}rassmann defectivity \`{a} la
  {T}erracini}, Le Matematiche \textbf{56} (2001), 245--255.

\bibitem{M2}
D.~R. Grayson and M.~E. Stillman, \emph{Macaulay2, a software system for
  research in algebraic geometry, available at
  http://www.math.uiuc.edu/macaulay2/}.

\bibitem{IK}
A.~Iarrobino and V.~Kanev, \emph{Power sums, {G}orenstein algebras, and
  determinantal loci}, vol. 1721, Springer-Verlag, Berlin, 1999.

\bibitem{St}
R.P. Stanley, \emph{{E}numerative {C}ombinatorics}, vol.~49, Cambridge Studies
  in Advanced Mathematics, Cambridge, 1997.

\bibitem{Strassen1973}
V.~Strassen, \emph{{V}ermeidung von {D}ivisionen}, J. Reine Angew. Math.
  \textbf{264} (1973), 184--202.

\bibitem{T2015}
Z.~Teitler, \emph{Sufficient conditions for {S}trassen's additivity
  conjecture}, Illinois J. Math. \textbf{59} (2015), no.~4, 1071--1085.

\bibitem{Terr}
A.~Terracini, \emph{Sulla rappresentazione delle coppie di forme ternarie
  mediante somme di potenze di forme lineari}, Ann. Mat. Pura Appl.
  \textbf{XXIV} (1915), no.~III, 91--100.

\end{thebibliography}
\end{document}